\documentclass[a4paper,10pt]{article}
\usepackage[a4paper]{geometry}
\usepackage[utf8]{inputenc}
\usepackage{amsmath}
\usepackage{amsfonts}
\usepackage{amssymb}
\usepackage{amsthm}
\usepackage{graphicx}
\usepackage{mathtools}
\usepackage{color}
\usepackage{enumerate}
\usepackage{marvosym}
\usepackage{textcomp}
\usepackage{tikz}
\usepackage[bookmarks=true]{hyperref}
\usepackage{placeins} 
\usepackage{epstopdf}   
\usepackage{todonotes}
\usepackage[toc,page]{appendix}

\usepackage{algorithmic}
\usepackage{listings}
\usepackage[makeroom]{cancel}
\definecolor{keywords}{RGB}{178,34,34}
\definecolor{comments}{RGB}{112,138,144}
\definecolor{background}{RGB}{255,245,238}
\usepackage{booktabs}

\newtheorem{theorem}{Theorem}[section]

\newtheorem{definition}[theorem]{Definition}
\newtheorem{proposition}[theorem]{Proposition}

\newcommand{\norm}[1]{\left\Vert#1\right\Vert}
\newcommand{\betrag}[1]{\left\vert#1\right\vert}
\newcommand{\ip}[2]{\left<#1,#2\right>}

\newcommand{\R}{\mathbb{R}}

\newcommand{\N}{\mathbb{N}}

\newcommand{\Landau}{\mathcal{O}}

\newcommand{\nL}{\par\medskip}

\title{Numerical Differentiation using local Chebyshev-Approximation}
\author{Stefan H. Reiterer}

\begin{document}

\maketitle

\begin{abstract}

\end{abstract}

\section{Motivation}

In applied mathematics, especially in optimization, functions are often only provided as so called "Black-Boxes" provided by software packages, or very complex 
algorithms, which make automatic differentation very complicated or even impossible. Hence one seeks the numerical approximation of the derivative.

Unfortunately numerical differentation is a difficult task in itself, and it is well known that it is numerical instable.
There are many works on this topic, including the usage of (global) Chebyshev approximations. Chebyshev approximations have the great property
that they converge very fast, if the function is smooth. Nevertheless those approches have several drawbacks, since in practice functions are
not smooth, and a global approximation needs many function evalutions.

Nevertheless there is hope. Since functions in real world applications are most times smooth except for finite points, corners or edges. This motivates
to use a local Chebyshev approach, where the function is only approximated locally, and hence the Chebyshev approximations still yields a fast approximation
of the desired function. We will study such an approch in this work, and will provide a numerical example.

Disclaimer: This work was done as private research during my study years and is not related to my current affiliation at all. No funding was recieved for this work.

\section{Chebyshev Polynomials and One Dimensional Functions}

First we recall some well known definitions. the following definitions and properties are taken from \cite{boyd}.
\begin{definition}[Chebyshev-Polynomials]
For $n \in \N_0$ the $n$-th Chebyshev polynomial $T_n(x): \,[-1,1] \rightarrow \R$ is given by $T_n(x) = \cos(n \arccos(x))$.
From this definition we see that the Chebyshev polynomial takes it's extrema at
\[ x_k := -\cos\left(\frac{k\pi}{n} \right) \text{ for } k = 0,\ldots,n. \]
The points $(x_k)_{k=0}^n$ are the so called Gauss-Lobatto Grid-points.\nL
\end{definition}
The best way to approximate functions via Chebyshev-polynomials lies in the following
\begin{proposition}[Chebyshev-Series \& Co]
Set $\omega(x):= \frac{2}{\pi}(1-x^2)^{-1/2}$, then \[ \ip{f}{g}_{\omega} := \int\limits_{-1}^1 f(x)g(x) \omega(x) \,dx \]
forms the (weighted) scalar product of the Hilbert-Space
\[L_{2,\omega} := \left\{ f:\, [-1,1]\rightarrow \R\, | \norm{f}_{L_{2,\omega}}=\sqrt{\ip{f}{f}_{\omega}} < \infty \right\}.\] 
Then for a function $f\in L_{2,\omega}$ the Chebyshev-Coefficients 
\[ a_n := \ip{f}{T_n}_{\omega}, \]
exists, and the Chebyshev-Series
\[ \frac{a_0}{2} + \sum\limits_{n=1}^{\infty} a_n T_n(x), \]
converges in the $L_{2,\omega}$ sense to $f$.
\end{proposition}
\begin{proof}
These well known facts follow from applying classical Fourier-theory to \nL $\tilde{f} = f(\cos(t)) \in L_2(0,\pi)$.
\end{proof}
Since Chebyshev-Polynomials are ``cosines in disguise'' other important properties from Fourier-theory carry over to Chebyshev-series,
like spectral convergence for smooth functions etc.

Another way to approximate a function is to use the Chebyshev-Interpolation Polynomial $C_n = \sum\limits_{k=0}^{n} = b_k T_k$,
which is uniquely defined by $f(x_k) = C_n(x_k)$ for $k=0,\ldots,n$. \nL
Although the Chebyshev-Series yields the best approximation there is a little known fact between the two approximation types, namely the
\begin{proposition}
Let $f:\,[-1,1] \rightarrow \R$ a function with
\[ f(x) = \frac{a_0}{2} + \sum\limits_{n=1}^{\infty} a_n T_n(x), \]
and \[ \frac{\betrag{a_0}}{2} + \sum\limits_{n=1}^{\infty} \betrag{a_n} < \infty. \]
Further let for $f_N(x) := \frac{a_0}{2} + \sum_{n=1}^{N} a_n T_n(x)$,
\[ E_T(N) := \sup_{x\in[-1,1]} \betrag{f(x) - f_N(x)}, \]
be the trunctation error, and
Then
\[ E_T(N) \leq \sum\limits_{n=N}^{\infty} \betrag{a_n}, \]
and for the interpolation error
\[ \sup_{x\in[-1,1]} \betrag{f(x) - C_N(x)} \leq 2 E_T(N). \]
Hence the penalty for using interpolation instead of truncation is at most a \textbf{factor of two}!
Additionally the coefficients $b_k$ of the interpolation polynomial $C_n$ are related to the exact coefficents $a_k$ by the identity
\begin{equation} \label{eq:aliasing}
 b_k = a_k + \sum\limits_{j=1}^{\infty} a_{k+ 2jn} + a_{-k+2jn}.
\end{equation}
That means the approximated coefficients differ from the exact coefficients by the aliasing-error, and hence the error vanishes with $\mathcal{O}(a_{n})$
\end{proposition}
\begin{proof}
See \cite[Thm. 6\& Thm. 21]{boyd}.
\end{proof}
Since with Chebyshev interpolation we are close to the realm of the DFT, additional methods from signal processing, like denoising could be considered to handle
numerical distortions.\nL
To work with a function $g:\, [a,b] \rightarrow \R$ a generalized Chebyshev interpolation for $g$ 
can be achieved by using a linear transformation $\varphi:\, [a,b] \rightarrow [-1,1]$, and interpolate the function $f = g\circ \varphi$. \nL

Since computing the derivative of a polynomial can be done exactly we could use the derivative of $C_N$ for some $N$, as the numerical derivative of $f$.
This means
\[ f^{\prime} \approx C_N^{\prime}. \]
Then how about the errors?
First recall the identity $T^{\prime}_n(x) = n U_{n-1}$.
Hence for the truncation error we have
\[ f^{\prime}(x) - f^{\prime}_N(x) = E_N^{\prime}(x) = \sum\limits_{n=N}^{\infty} n a_n U_{n-1}(x). \]
Considering that $U_n$ has it's extrema at $\pm 1$ with $U_{n-1}(\pm 1) = (\pm 1)^n n$,
we immediately see that
\[ \sup_{x\in[-1,1]}\betrag{f^{\prime}(x) - f^{\prime}_N(x)} \leq \sum\limits_{n=N}^{\infty} n \betrag{a_n} \sup_{x\in[-1,1]} \betrag{U_{n-1}(x)} \leq \sum\limits_{n=N}^{\infty} n^2 \betrag{a_n} = 
\Landau(n^2 a_n). 
\]
For the interpolation polynomial $C_N$ we get by using the aliasing identity \eqref{eq:aliasing} and rearranging terms like in the proof of \cite[Thm. 21]{boyd}
the error estimation
\[ \betrag{f^{\prime}(x) - C^{\prime}_N(x)} \leq 2 \sum\limits_{n=N}^{\infty} n^2 \betrag{a_n} = 
\Landau(2 n^2 a_n). 
\]
Hence we have the
\begin{proposition} \label{prop:derivative}
For $f:\, [-1,1] \rightarrow \R$ smooth we have the error estimation
\[ \sup_{x\in[-1,1]} \betrag{f^{\prime}(x) - C^{\prime}_N(x)} \leq 2 E_T^{\prime}(N) \Landau(2 n^2 a_n). \]
Hence the penalty for using interpolation instead of truncation when differentiating is at most a \textbf{factor of two}!
\end{proposition}

%

\section{Practical Considerations in the 1D Case and Generalizations}

In this section consider a function $f:\, [a,b] \rightarrow \R$, which is continuous and piecewise smooth, but not differentiable in  
a set of finite points $(y_k)_{k\in I} \subset [a,b]$ (with $I$ finite and allowed to be empty).

If we interpolate now the function $f$ on the interval $[a,b]$, we know from Fourier-Theory, that the error will converge only slowly to zero.
We also note that it is easy to see that (from variable transformations) that error estimation depends on the length of the interval $[a,b]$.
with the factor $(\frac{b-a}{2})^N$. Hence if one restricts the function $f$ only on a small interval, one achieves more precise interpolation results, with
lesser degrees of the interpolation polynomial.
If we restrict the function $f$ to some interval $[c,d] \subset [a,b]$, with
\begin{enumerate}
 \item $d-c < 1$ and
 \item $y_i \not\in (c,d)$ for $i \in I$,
\end{enumerate}
then the sequence of the Chebyshev interpolation polynomials $(C_N)_{N\in\N}$ converges very fast (with spectral convergence) 
and uniformly to the original function in the interval $[c,d]$.

From Proposition~\ref{prop:derivative} we know that this then also applies to $C_N^{\prime}$, since the error decays with $\Landau(n^2 a_n)$.
Nevertheless we also see that for practical use it is important
to ensure the smoothness on the observed interval $[c,d]$, because of the term $n^2$. This leads to the following definition
\begin{definition}
Let $f:\, [a,b] \rightarrow \R$ be continuous and piecewise smooth, 
then we call the Chebyshev interpolation polynomial $C_{N,[c,d]}:\, [c,d] \rightarrow \R$ of $f_{|[c,d]}$ \textbf{the local smooth Chebyshev polynomial}
iff \[(y_k)_{k\in I} \cap (c,d) = \emptyset.\]
\end{definition}
This motivates the following algortithm to compute derivatives of a function $f$ (if the points $Y := (y_k)_{k\in I}$ are known):
\begin{lstlisting}
# input: function f, and point x, estimated lenght h
         set of disallowed points Y, number of interpolation points N
def compute_derivative(f,x,h,Y):
  if x not in Y:
  # classical derivative
     c = x-h; d= x+h
     # ensure that we are smooth on interval [c,d] 
     while c < a and b < d and intersection(Y,[c,d]) != {}:
        h = make_h_smaller(h)
        c = x-h; d= x+h
  
     CN = compute_interpolation(f,c,d,N)
     CN_prime = derive_polynomial(CN)
     return CN_prime(x)
  
  else: # x in Y
    # subgradient
    c[0] = x-h; d[0] = x
    c[1] = x; d[1] = x+h
    while all([c[i] < a and b < d[i] and \
       intersection(Y,(c[i],d[i])) != {} for i in [0,1]]):
       h = make_h_smaller(h)
       c[0] = x; d[0] = x+h
       c[1] = x-h; d[1] = x
  
     for i in [0,1]:
       CN[i] = compute_interpolation(f,c[i],d[i],N)
       CN_prime[i] = derive_polynomial(CN[i])
     return [CN_prime[0](x),CN_prime[1](x)]
\end{lstlisting}
As one can see we also included more general derivatives based on directional derivatives like subgradients. It also would be possible to define a weak derivative witch is
motivated by Chebyshev convergence theory defined by $f^{\prime}(x) := \frac{1}{2}(f^{\prime}(-x) + f^{\prime}(+x))$. \nL
A question which still remains, is what to do when the set $Y$ is unknown. This is still an open problem, but since one can observe the speed of the decay of the 
coefficients $a_n$ (or $b_n$) it is possible to locate non differentiable points. There are several works on this topic from spectral theory. \nL
Now to the $M$ dimensional case:
Consider a functional $f:\Omega \subseteq \R^M \rightarrow \R$, which is continuous and piecewise smooth on the domain $\Omega$.
We can compute the directional derivatives in a direction $h \in \R^M$ in a point $x \in \Omega$ by deriving the one dimensional
function $g(t):= f(x+th)$. Hence the methods from the one dimensional case, can be carried over to the $M$ dimensional setting.
\section{Examples}
as first we define the function $f_1$ given by (see also Figure~\ref{fig:func}):
\begin{equation*}
f_1(x) = \begin{cases}
          x^4 \text{ if } x > 0, \\
          0 \text{ else.}
         \end{cases}
\end{equation*}
First we compare computing the first derivative with local Chebyshev approximation with finite differences given by
\[ f^{\prime}(x) \approx \frac{f(x+h) - f(x-h)}{2h} =: f_h(x), \]
by computing the values at $x = 0$ and $x = 0.5$.
See Table~\ref{tab:errors1} for results. As we can observe the errors of $f_h$ and $C_{3,[x\pm h]}$ are equal. This comes without surprise, since the
derivative of the $3$rd Chebyshev polynomial is exactly the central difference quotient.
\begin{table} \label{tab:errors1}
\begin{tabular}{ccccc}
$x$   &  $h$  & $f_h$    & $C_{3,[x\pm h]}$ & $C_{5,[x\pm h]}$   \\
$0.5$ &  $1e-3$ & $2e-6$   & $2e-6$           & $2e-14$            \\
      &  $1e-4$ & $1.9e-8$ & $1.9e-8$         & $2.6e-13$          \\
      & $1e-5$ &$1.9e-10$& $1.9e-10$        & $1.8e-12$          \\
$0.0$ & $1e-3$  & $5e-10$ & $5e-10$          & $1.4e-10$          \\
      & $1e-4$ & $5e-13$ & $5e-13$          & $2.6e-13$          \\
      & $1e-5$ & $5e-16$ & $5e-16$          & $1.5e-16$          \\
\end{tabular}  
\caption{Errors for $f_1$}
\end{table}
Next we will modify $f_1$ in the following way: For some $\varepsilon > 0$, and a randomly standard normal distributed variable $X$ we define $f_2$ by
\begin{equation*}
f_2(x) = f_1(x) + \varepsilon X
\end{equation*}
and then try to differentiate it at $0.5$. the results can be seen at Table~\ref{tab:errors2}.
\begin{table} \label{tab:errors2}
\begin{tabular}{ccccccc}
$x$   &  $h$  & $f_h$      & $C_{3,[x\pm h]}$ & $C_{5,[x\pm h]}$  & $C_{7,[x\pm h]}$ \\
$0.5$ &  $1e-1$ & $2e-2$   & $2e-2$           & $5.5e-9$          & $4.2e-8$         \\
      &  $1e-2$ & $2.4e-6$ & $2e-4$         & $4e-7$              & $8e-7$           \\
      & $1e-3$ &  $2.4e-6$ & $2e-6$        & $1.8e-6$             & $7.3e-6$         \\
      & $1e-4$  & $1e-6$   & $4e-6$          & $2e-5$             & $2e-5$           \\
      & $1e-5$ & $2.8e-5$  & $9e-5$          & $2e-5$             & $7e-4$           \\
      & $1e-6$ & $7e-5$    & $1e-3$          & $4e-5$             & $2e-3$           \\
\end{tabular}  
\caption{Errors for $f_2$}
\end{table}
The errors are maximal values after taking some samples (\textbf{There is some space for improvement here, e.g. better statistics, but we get the picture...}) 
We Observe the following behavior: While $f_h$ and $C_{3,[x\pm h]}$ behave like expected,
we see that making $h$ smaller, does not make the approximation quality of $C_{5,[x\pm h]}$ and $C_{7,[x\pm h]}$ better, but even worse.
This can be explained by the fact, that for smooth $f$, we know that the function behaves like it's Taylor polynomial of lower order, while the
higher order terms are neglectable. This also means that computing a higher order Chebyshev polynomial does not make sense, and we even get into more trouble.
hence when choosing a suitable $h$ for the Chebyshev approximation should be not too small, and yet not too big.
\begin{figure} \label{fig:func}
\begin{center}
 \includegraphics[scale=0.25]{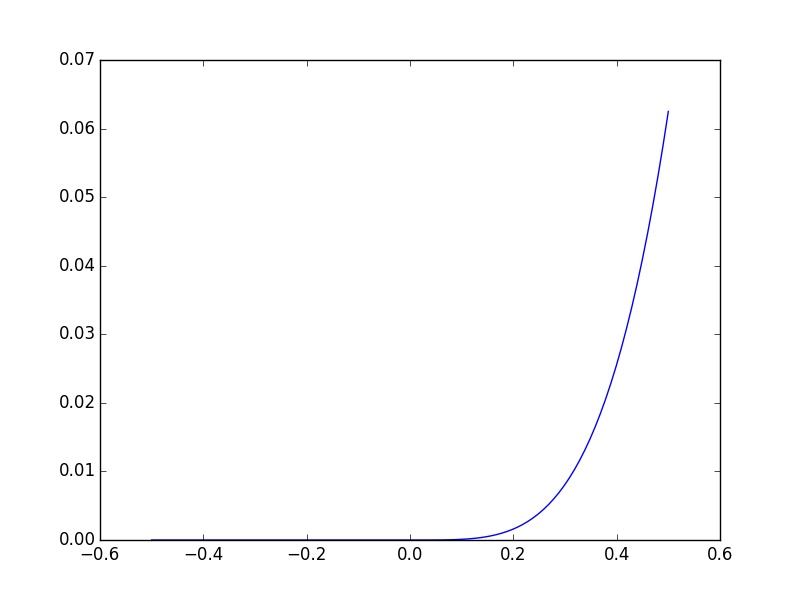}
\end{center}

\caption{Function $f_1$}
\end{figure} 

Finally we try the Local Chebyshev Method on the 2D Rosenbrock function
\begin{equation*}
R_{a,b}(x) = (a-x_0)^2 + b(x_1-x_0^2)^2, 
\end{equation*}
with parameters $a = 1$ and $b = 100$, and try to find the minimum with help of gradient methods. The argmin is $x_0 = (1,1)$
Additionally we add some disturbances $\delta(x)$ and $\varepsilon(x)$ in the interval $[-1e-6,1e-6]$. 
While $\delta$ is some randomly normal distributed variable, $\varepsilon$ is a jump function which changes it's sign (which is fatal for finite differences).
As optimization method we used Steepest Descent with Armijo-rule and $\norm{\nabla f} < 1e-3$ as terminating criterion. 
We used $h=1e-6$ for the finite difference method, and order $5$ polynomials with $h=1e-4$ for the Local Chebyshev Method.
and 
The number of iterations and the computed results (rounded up to 3 digits) can be seen in Table~\ref{tab:iterations}.
The results suggest, that computing the gradient with the Local Chebyshev Method makes the algorithm more robust. Also it was observed that
the stepsize $h$ should be not too small.
\begin{table} \label{tab:iterations}
\begin{tabular}{cccc}
Function & Method & Iteration Numbers & Result \\
$R_{a,b}$& Exact Gradient & $1427$    & $(0.999,  0.998)$ \\
         & Finite Differences & $1427$ &$(0.999,  0.998)$ \\
         & Local-Chebyshev    & $1428$ &$(0.999,  0.998)$  \\
$R_{a,b}+\delta$ & Finite Differences & $19999$ (interrupted) & $(0.986,  0.973)$ \\
                 & Local-Chebyshev    & $1090$                 & $( 0.996,   0.992)$ \\
$R_{a,b}+\varepsilon$ & Finite Differences & $19999$ (interrupted) & $(0.486,  0.234)$  \\
                      & Local-Chebyshev    & $1565$                & $(0.999  0.998)$ 
\end{tabular}  
\caption{Iteration numbers for optimization}
\end{table}

\end{document}